%% file: main.tex
\documentclass[11pt]{article}
\usepackage[margin=28mm]{geometry}
\usepackage{amsmath,amssymb,amsthm}
\usepackage{booktabs}
\usepackage{array}
\usepackage{needspace}
\usepackage{tikz}
\usepackage[colorlinks=true,linkcolor=blue!60!black,citecolor=blue!60!black,urlcolor=blue!60!black]{hyperref}

\definecolor{pathblue}{RGB}{40,90,200}
\definecolor{pathorange}{RGB}{220,120,30}
\definecolor{tileA}{RGB}{221,232,248}
\definecolor{tileB}{RGB}{250,235,215}
\definecolor{tileC}{RGB}{235,240,232}

\newtheorem{theorem}{Theorem}[section]
\newtheorem{lemma}[theorem]{Lemma}
\newtheorem{corollary}[theorem]{Corollary}

\theoremstyle{remark}
\newtheorem{remark}[theorem]{Remark}
\theoremstyle{definition}

\newcommand{\floor}[1]{\left\lfloor #1 \right\rfloor}
\newcommand{\core}{\mathsf{C}}
\newcommand{\Zt}{\mathsf{Z}}
\newcommand{\Kt}{\mathsf{K}}
\newcommand{\Wt}{\mathsf{W}}
\newcommand{\Xt}{\mathsf{X}}

\title{Farey Structure in Modulo Krinkle Tilings:\\
Mediant Splicing and Generation of Prototiles from a Single Edge}
\author{Mikihiro Fujiwara\\[2pt] \small Independent Researcher, Tokyo, Japan}
\date{\today}

\begin{document}
\maketitle

\begin{abstract}
The Modulo Krinkle tilings of Imura (arXiv:2506.07638) form a family of non-periodic,
spiral monohedral tilings parametrized by a reduced fraction $m/k$ and an integer $t\ge 2$.
We show that the Farey sum (mediant) $(m_1+m_2)/(k_1+k_2)$ of two Farey-adjacent parameters
is realized by an exact geometric operation on prototiles: the lower boundary path of the
$(m_1+m_2,k_1+k_2)$-prototile is obtained by concatenating the parents' lower paths after an
edge-length-preserving progressive rotation (\emph{fan-twist}) of their edges.
Conversely, every prototile admits exactly one fan-twist splice decomposition---no
non-adjacent parameters ever splice---the cut position being $k_1=m^{-1}\bmod k$,
and the recursion descends the Stern--Brocot tree to a single unit edge.
The combinatorial core of the operation is the classical standard factorization of Christoffel
words; the contribution here is its exact edge-isometric realization on circular direction systems
and the resulting structure theory for the Modulo Krinkle family, including the recently
introduced variants: we prove a \emph{separation theorem} stating that every variant prototile
is the common recursively-generated core plus finitely many direction-invariant decoration
edges. As a corollary of Imura's spiral-arm count formula, the two Farey parents are visible
in the offset-free tiling itself: the numbers of counterclockwise and clockwise spiral arms
are $t\,k_1$ and $t\,k_2$.
\end{abstract}

\section{Introduction}\label{sec:intro}

Imura~\cite{Imura2025} introduced a family of non-periodic monohedral tilings of the plane,
the \emph{Modulo Krinkle} tilings, constructed from elementary modular arithmetic:
a reduced fraction $m/k$ and an integer $t\ge 2$ determine a $(2k+2)$-gon prototile whose
rotated copies tile the plane with a characteristic staggered rotational structure.
A follow-up paper~\cite{Imura2026} reformulates the family as spiral tilings---continuing a
line of constructions going back to Voderberg and to Gailiunas'
spirals~\cite{Gailiunas2000,StockWichmann2000,GruenbaumShephard1987}---and introduces three
variants (called $\mathrm{m}$, $\mathrm{p}$, $\mathrm{q}$) whose prototiles split into two
mirror-image parts.

\begin{figure}[t]
\centering
\input{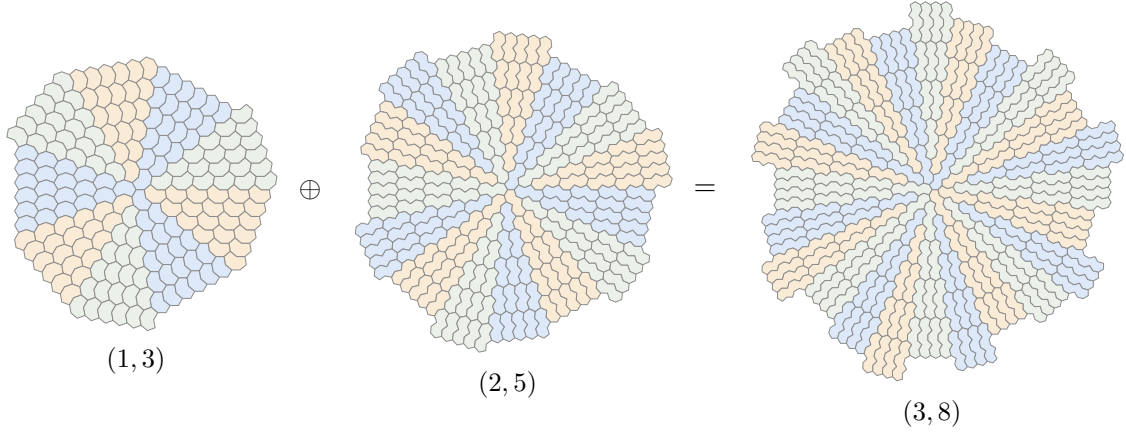}
\caption{Patches of three Modulo Krinkle tilings for $t=3$: the $(1,3)$- and
$(2,5)$-tilings and the tiling of their mediant $(3,8)$. Each tiling is monohedral: all
tiles are rotated copies of a single $(2k+2)$-gon prototile. This note shows that the
parameter operation $(1,3)\oplus(2,5)=(3,8)$ is realized by an exact geometric operation
on the prototiles (Theorem~\ref{thm:splice}), and that the two parents stay visible in
the child tiling as its $t\,k_1=9$ counterclockwise and $t\,k_2=15$ clockwise spiral
arms, where $k_1=3$ and $k_2=5$ are the parents' denominators
(Section~\ref{sec:arms}).}
\label{fig:tilings}
\end{figure}

This note studies the parameter space of the family through the map
$(m_1,k_1),(m_2,k_2)\mapsto(m_1+m_2,k_1+k_2)$, i.e.\ the \emph{Farey sum} or
\emph{mediant} of the two fractions; Figure~\ref{fig:tilings} shows a first example.
Our results are as follows.

\begin{enumerate}
\item \textbf{Index correspondence (Lemma~\ref{lem:index}).}
If $m_1/k_1<m_2/k_2$ are Farey-adjacent ($m_2k_1-m_1k_2=1$), the direction indices
$s_j = jm\bmod k$ of the child $(m,k)=(m_1+m_2,k_1+k_2)$ are given by closed-form
expressions in the parents' indices; the divisibility contained in these formulas is
exactly equivalent to the use of Farey adjacency.
\item \textbf{Mediant splicing (Theorem~\ref{thm:splice}).}
The child's lower boundary path coincides \emph{vertex by vertex} with the concatenation
of the parents' lower paths after a progressive rotation of their edges
(rotation angle affine in the edge number; all edges keep unit length).
Combinatorially this is the standard factorization of Christoffel
words~\cite{BorelLaubie1993,BLRS2008}; geometrically it is an exact edge-isometric
realization of that factorization on the circular direction system $\{e^{2\pi i d/n}\}$.
\item \textbf{Unique decomposition and generation
(Theorem~\ref{thm:decompose}, Corollary~\ref{cor:generation}).}
Every prototile admits exactly one fan-twist splice decomposition: uniqueness holds not
merely among Farey-adjacent pairs but among \emph{all} cut positions and parameters, the
divisibility in the angle-matching conditions forcing adjacency. The cut position is the
modular inverse $k_1=m^{-1}\bmod k$. The recursion descends the Stern--Brocot tree and
always terminates at a single unit edge, so the lower paths of the whole family---and
with them the prototiles---are generated from one edge by iterated fan-twist splicing.
\item \textbf{Separation of generation and decoration (Theorem~\ref{thm:separation}).}
All seven boundary templates of \cite{Imura2026} (original and variants
$\mathrm{m},\mathrm{p},\mathrm{q}$, two parts each) are the \emph{same} recursively
generated core plus finitely many \emph{decoration edges} whose angles
($0$ and $2\pi/t$) are invariant across the parameter change. Consequently the variants
are covered by the same generation theory. The theorem is stated for arbitrary
templates with system-invariant decorations; the admissible decoration alphabet has
exactly $t$ letters, and the family of closed template pairs satisfying our theorems is
infinite---which templates tile the plane is a separate question
(Problem~3 of Section~\ref{sec:open}).
\item \textbf{Spiral arms (Corollary~\ref{obs:arms}).}
Combining Theorem~\ref{thm:decompose} with the arm-count formula of \cite{Imura2026},
the numbers of counterclockwise and clockwise spiral arms of the offset-free
$(m,k,t)$-tiling are $t\,k_1$ and $t\,k_2$: the two Farey parents are directly visible
in the tiling.
\end{enumerate}

\begin{figure}[t]
\centering
\input{figs/fig-tree}
\caption{The family as a Stern--Brocot tree ($t=3$): each node carries the prototile of
the fraction $m/k$, and the thin edges join a fraction to its two Farey parents (blue:
left parent, orange: right parent). The two seeds $0/1$ and $1/1$ share the same unit
rhombus. Every prototile is obtained from that rhombus by iterated mediant splicing
along the tree (Corollary~\ref{cor:generation}); the highlighted node $(3,8)$, with
parents $(1,3)$ and $(2,5)$, is the running example of this paper. Nodes of depth $4$
other than $3/8$ are omitted.}
\label{fig:tree}
\end{figure}
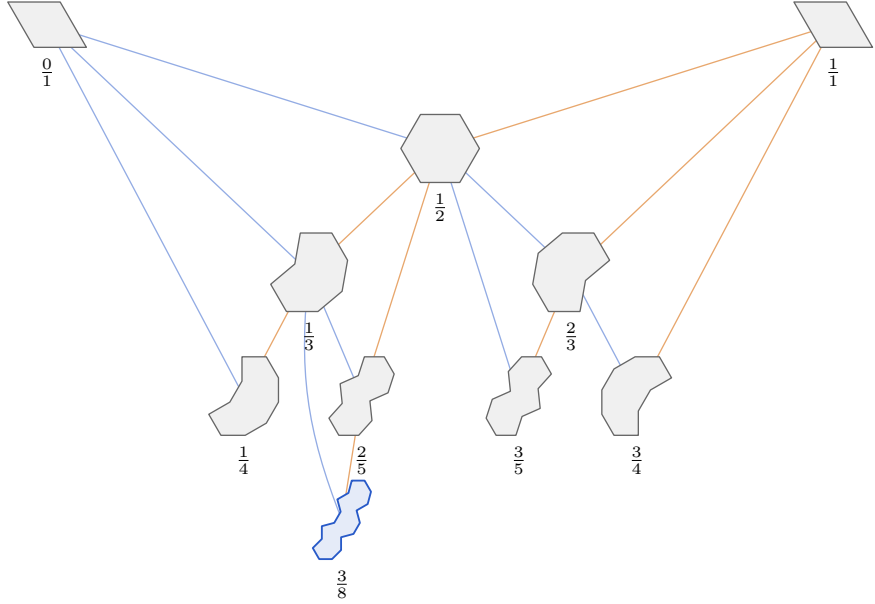

Figure~\ref{fig:tree} summarizes the resulting picture: the family of prototiles,
arranged by the value of $m/k$, is the Stern--Brocot tree grown from a single unit
rhombus.

\paragraph{Related work.}
The standard factorization of Christoffel words is classical~\cite{BorelLaubie1993}; see
\cite{BLRS2008} for a systematic treatment, including its geometric meaning on lattice
paths (the factorization point is the lattice point closest to the line segment).
Christoffel words also appear in discrete geometry as digital straight lines
\cite{KletteRosenfeld2004}, where Stern--Brocot descent underlies subsegment and
recognition algorithms~\cite{SaidLachaud2011}, and in the study of polyominoes that tile
the plane by translation~\cite{BlondinMasse2009}.
The prototile-level mediant correspondence itself was observed informally by Imura
before the present work: she noted that the $(5,12)$-prototile resembles a ``hybrid''
of the $(2,5)$- and $(3,7)$-prototiles, matching
$\tfrac{2}{5}\oplus\tfrac{3}{7}=\tfrac{5}{12}$ in the Stern--Brocot
tree~\cite{ImuraPost2025}. The observation appears in neither \cite{Imura2025} nor
\cite{Imura2026} and reached the author only after the present work was completed. What is new here is the precise formulation
and proof of this correspondence: the identification of its combinatorial core with
the Christoffel standard factorization, the lift of that factorization to an
\emph{exact} edge-isometric operation on the circular direction systems---uniformly
in the variant templates---and the resulting uniqueness, generation, and separation
theorems.

\section{Preliminaries}\label{sec:prelim}

Fix an integer $t\ge 2$. For a reduced fraction $0< m/k< 1$ (we also allow the formal
\emph{seeds} $(m,k)=(0,1)$ and $(1,1)$), put
\[
n = t\,k, \qquad
v_n(d) = \bigl(\cos\tfrac{2\pi d}{n},\ \sin\tfrac{2\pi d}{n}\bigr).
\]
Define the \emph{direction indices}
\[
s_j = j\,m \bmod k \qquad (j=0,1,\dots,k-1),
\]
a permutation of $\{0,\dots,k-1\}$. The \emph{lower word} is
$\ell=(s_0,\dots,s_{k-1},k)$ and the \emph{upper word} $u$ is $\ell$ with its first and
last entries exchanged. The \emph{lower path} $L$ starts at the origin $P$ and appends the
unit edges $v_n(\ell_0),\dots,v_n(\ell_k)$; the \emph{upper path} $U$ does the same with
$u$. Both end at the same point $Q$, and the closed curve $L\cup U$ bounds a simple
$(2k+2)$-gon, the \emph{prototile}; rotated copies of it tile the plane
\cite[Thm.~6.11]{Imura2025}. We call the first $k$ edges of $L$ (directions
$s_0,\dots,s_{k-1}$) the \emph{core} and the last edge (direction $k$, angle $2\pi/t$)
the \emph{terminal edge}.

\begin{figure}[t]
\centering
\input{figs/fig-proto}
\caption{The $(2,5)$-prototile ($t=3$): lower path $L$ (blue), upper path $U$ (gray),
terminal edges dashed.}
\label{fig:proto}
\end{figure}
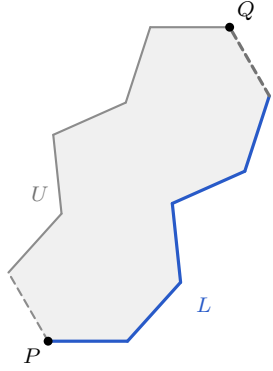

\paragraph{Running example.}
For $(m,k)=(2,5)$ and $t=3$ we get $n=15$ and
\[
s=(0,2,4,1,3), \qquad \ell=(0,2,4,1,3,5), \qquad u=(5,2,4,1,3,0):
\]
the resulting $12$-gon is shown in Figure~\ref{fig:proto}, and its tiling is the middle
patch of Figure~\ref{fig:tilings}. The pair $(1,3),(2,5)$ is Farey-adjacent
($2\cdot3-1\cdot5=1$) with mediant $(3,8)$; this triple recurs as the running example in
Figures~\ref{fig:tilings}, \ref{fig:tree}, \ref{fig:splice} and~\ref{fig:cut}.

The \emph{carry word} of $(m,k)$ is the binary word
\[
c(m,k) = (c_0,\dots,c_{k-1}), \qquad
c_j = \floor{\tfrac{(j+1)m}{k}} - \floor{\tfrac{jm}{k}} \in\{0,1\},
\]
which records whether the increment $s_{j+1}-s_j$ equals $m$ (letter $0$) or $m-k$
(letter $1$). As a lattice-path code, $c(m,k)$ is the lower Christoffel word of slope
$m/(k-m)$ \cite{BLRS2008}. For the running example, $c(2,5)=(0,0,1,0,1)$.

Two reduced fractions $m_1/k_1<m_2/k_2$ are \emph{Farey-adjacent} if
$m_2k_1-m_1k_2=1$; their \emph{mediant} is $(m_1+m_2)/(k_1+k_2)$, which is then
automatically reduced. Throughout, $(m,k)=(m_1+m_2,k_1+k_2)$ and we write
$s^{(1)},s^{(2)},s$ for the direction indices of $(m_1,k_1),(m_2,k_2),(m,k)$, and
$n_1=t k_1$, $n_2=t k_2$, $n=t k$; note the additivity $n_1+n_2=n$.

\section{The index correspondence lemma}\label{sec:lemma}

\begin{lemma}[Index correspondence]\label{lem:index}
Let $m_1/k_1<m_2/k_2$ be Farey-adjacent and $(m,k)=(m_1+m_2,k_1+k_2)$. Then
\begin{align}
s_j &= \frac{k\,s^{(1)}_j + j}{k_1}, &
\floor{\tfrac{jm}{k}} &= \floor{\tfrac{jm_1}{k_1}}
&& (0\le j\le k_1-1), \tag{$\ast$1}\label{eq:first}\\[2pt]
s_{k_1+i} &= 1 + \frac{k\,s^{(2)}_i - i}{k_2}, &
\floor{\tfrac{(k_1+i)m}{k}} &= m_1 + \floor{\tfrac{im_2}{k_2}}
&& (0\le i\le k_2), \tag{$\ast$2}\label{eq:second}
\end{align}
where in \eqref{eq:second} the first formula is asserted for $0\le i\le k_2-1$.
In particular both numerators are divisible as claimed. Moreover
\begin{equation}
k_1 m - m_1 k \;=\; m_2k_1-m_1k_2 \;=\; 1. \tag{$\star$}\label{eq:star}
\end{equation}
\end{lemma}

\begin{proof}
The key is to use Farey adjacency in the form $m_2=(m_1k_2+1)/k_1$.

\emph{First half.} Fix $0\le j\le k_1-1$ and put $q=\floor{jm_1/k_1}$, so that
$jm_1=qk_1+s^{(1)}_j$. Then
\[
jm_2=\frac{jm_1k_2+j}{k_1}=\frac{(qk_1+s^{(1)}_j)k_2+j}{k_1}
= qk_2+\frac{s^{(1)}_jk_2+j}{k_1},
\qquad\text{hence}\qquad
jm = qk + \frac{k\,s^{(1)}_j+j}{k_1}.
\]
It remains to check that the last fraction is the remainder of $jm$ modulo $k$:
(i) \emph{divisibility}: since $k\equiv k_2\pmod{k_1}$ and
$s^{(1)}_j\equiv jm_1 \pmod{k_1}$,
\[
k\,s^{(1)}_j + j \equiv k_2 s^{(1)}_j + j \equiv j(k_2m_1+1) = j\,k_1m_2 \equiv 0 \pmod{k_1};
\]
(ii) \emph{range}: $0\le (k\,s^{(1)}_j+j)/k_1 \le (k+1)(k_1-1)/k_1 < k$
(using $s^{(1)}_j,\,j\le k_1-1$).
Uniqueness of the remainder gives both formulas in \eqref{eq:first}.

\emph{Second half.} Symmetrically use $m_1=(m_2k_1-1)/k_2$. First,
$k_1m-m_1k = m_2k_1-m_1k_2 = 1$, which is \eqref{eq:star}. With
$q'=\floor{im_2/k_2}$ we get
$im_1=(im_2k_1-i)/k_2=q'k_1+(s^{(2)}_ik_1-i)/k_2$, so that
$im = im_1+im_2 = q'k + (k\,s^{(2)}_i-i)/k_2$; hence, using \eqref{eq:star},
\[
(k_1+i)m = k_1m + im = (m_1+q')\,k + 1 + \frac{k\,s^{(2)}_i-i}{k_2}.
\]
(i) \emph{divisibility}: $k_1s^{(2)}_i - i\equiv i(k_1m_2-1)= i\,m_1k_2\equiv 0\pmod{k_2}$;
(ii) \emph{range}: for $i=0$ the value is $1$; for $i\ge 1$ we have $s^{(2)}_i\ge1$
(as $\gcd(m_2,k_2)=1$), so the numerator is $\ge k-i>0$, and from above
$1+(k(k_2-1)-i)/k_2 \le 1+k-k/k_2 < k$. This proves \eqref{eq:second} for
$i\le k_2-1$; for $i=k_2$ both sides of the floor identity equal $m$.
Farey adjacency is used exactly once in each divisibility step.
Part~(2) of Theorem~\ref{thm:decompose} below turns this observation into a precise
converse: the integrality of the closed forms \emph{forces} adjacency.
\end{proof}

\section{Mediant splicing}\label{sec:splice}

\begin{theorem}[Mediant splicing]\label{thm:splice}
Let $m_1/k_1<m_2/k_2$ be Farey-adjacent, $t\ge2$ common, and $n=t(k_1+k_2)$.
Then the lower path of the $(m_1+m_2,\,k_1+k_2)$-prototile coincides, vertex by vertex,
with the path obtained by concatenating from the origin:
\begin{itemize}
\item the edges $j=0,\dots,k_1-1$ of the lower path of $(m_1,k_1)$, edge $j$ rotated by
$+\,j\cdot 2\pi/(k_1 n)$;
\item the edges $i=0,\dots,k_2-1$ of the lower path of $(m_2,k_2)$, edge $i$ rotated by
$+\,2\pi/n - i\cdot 2\pi/(k_2 n)$;
\item the terminal edge (angle $2\pi/t$, identical in all three tilings).
\end{itemize}
All edges keep unit length; only their directions change, by an angle affine in the
edge number (a \emph{fan-twist}).
\end{theorem}

\begin{proof}
The angle of the child's core edge $j<k_1$ is, by Lemma~\ref{lem:index} and $k/n=1/t$,
\[
\frac{2\pi s_j}{n}
= \frac{2\pi\,(k\,s^{(1)}_j + j)}{k_1 n}
= \frac{2\pi s^{(1)}_j}{n_1} + j\cdot\frac{2\pi}{k_1 n},
\]
whose first term is precisely the angle of edge $j$ of the first parent. Similarly, for
edge $k_1+i$,
\[
\frac{2\pi s_{k_1+i}}{n}
= \frac{2\pi s^{(2)}_i}{n_2} + \frac{2\pi}{n} - i\cdot\frac{2\pi}{k_2 n}.
\]
The terminal edges have angle $2\pi k/n = 2\pi k_1/n_1 = 2\pi k_2/n_2 = 2\pi/t$.
All edges are unit and agree direction-by-direction, so the two vertex sequences starting
at the origin coincide exactly. Finally, the spliced path \emph{is} the canonical lower
path of the child, so simplicity of the prototile and the tiling property are inherited
from \cite{Imura2025} with no further argument.
\end{proof}

\begin{corollary}[Standard factorization]\label{cor:standard}
For Farey-adjacent parameters, $c(m_1+m_2,\,k_1+k_2) = c(m_1,k_1)\,c(m_2,k_2)$.
Since both factors are Christoffel words and a Christoffel word factors into two
Christoffel words in a unique way, this identity recovers the standard factorization of
\cite{BorelLaubie1993}.
\end{corollary}

\begin{proof}
Substitute the floor identities of Lemma~\ref{lem:index} into
$c_j=\floor{(j+1)m/k}-\floor{jm/k}$. For $j\le k_1-2$ take the difference of two
instances of \eqref{eq:first}; at the boundary $j=k_1-1$ combine \eqref{eq:second} with
$i=0$ and \eqref{eq:first} with $j=k_1-1$; for $j=k_1+i$ take differences of
\eqref{eq:second} (the last letter uses the case $i=k_2$).
\end{proof}

\begin{remark}[Offset variant]\label{rem:offset}
Imura's \emph{offset} variant replaces $n=tk$ by $n=2(tk-m)$. Lemma~\ref{lem:index} does
not mention $n$ and holds verbatim, and $n$ remains additive:
$2(tk_1-m_1)+2(tk_2-m_2)=2(tk-m)$. However, the \emph{affine-in-index} form of the
fan-twist in Theorem~\ref{thm:splice} is specific to $n\propto k$: with the offset one has
$k n_1 - k_1 n = 2(mk_1-m_1k)=2$, and the rotation amounts acquire correction terms,
\[
\mathrm{rot}_j = 2\pi\Bigl(\frac{j}{k_1 n} + \frac{2\,s^{(1)}_j}{k_1 n n_1}\Bigr),
\qquad
\mathrm{rot}_{k_1+i} = 2\pi\Bigl(\frac{1}{n} - \frac{i}{k_2 n}
 - \frac{2\,s^{(2)}_i}{k_2 n n_2}\Bigr),
\]
while the terminal edge, no longer direction-invariant, rotates by $4\pi/(n n_1)$
relative to the first parent. The general statement---edge rotation
$=2\pi s_j/n - 2\pi s^{(1)}_j/n_1$, always in closed form by
Lemma~\ref{lem:index}---holds for both conventions.

In fact both conventions belong to the affine family $n=\alpha k+\beta m$
($\alpha,\beta\in\mathbb{Z}$ fixed for the family; Imura's two are
$(\alpha,\beta)=(t,0)$ and $(2t,-2)$). For any such convention $n$ remains additive,
Farey adjacency forces the \emph{constant} residues $k\,n_1-k_1 n=-\beta$ and
$k\,n_2-k_2 n=\beta$, and consequently
\[
\mathrm{rot}_j = 2\pi\Bigl(\frac{j}{k_1 n} - \frac{\beta\,s^{(1)}_j}{k_1 n n_1}\Bigr),
\qquad
\mathrm{rot}_{k_1+i} = 2\pi\Bigl(\frac{1}{n} - \frac{i}{k_2 n}
 + \frac{\beta\,s^{(2)}_i}{k_2 n n_2}\Bigr),
\]
with terminal-edge rotation $-2\pi\beta/(n n_1)$: the entire deviation from the
fan-twist is governed by the single integer $\beta$, and the rotation is affine in
the edge number precisely when $\beta=0$. This is a statement about boundary paths
only; for the tileability of conventions beyond Imura's two see
Section~\ref{sec:open}.
\end{remark}

\begin{figure}[t]
\centering
\input{figs/fig-splice}
\caption{Mediant splicing $(1,3)\oplus(2,5)\to(3,8)$ ($t=3$): the first $3$ edges of the
child (blue) are the fan-twisted lower path of $(1,3)$; the next $5$ edges (orange) the
fan-twisted lower path of $(2,5)$; the dashed terminal edge is common.}
\label{fig:splice}
\end{figure}
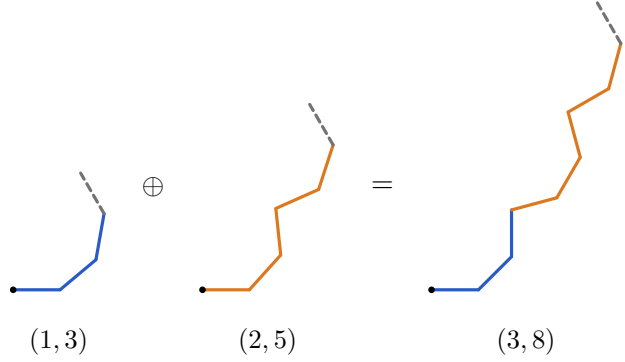

\section{Unique decomposition and generation}\label{sec:decompose}

Geometrically, decomposing a prototile means cutting its lower path and undoing the
fan-twist of Theorem~\ref{thm:splice}. The following theorem states that this succeeds
at exactly one cut position and for exactly one choice of parameters---adjacency is not
an assumption on the candidate decompositions but a consequence of the geometry---and
identifies the parameters as the Farey parents of $m/k$.

\begin{theorem}[Unique decomposition]\label{thm:decompose}
Let $(m,k)$ be reduced with $k\ge2$ and $n=tk$, and define the \emph{canonical pair}
\[
k_1 = m^{-1} \bmod k, \qquad m_1 = \frac{mk_1-1}{k}, \qquad
(m_2,k_2)=(m-m_1,\ k-k_1).
\]
\begin{enumerate}
\item[\emph{(1)}] \emph{(Farey parents.)} The canonical pair is Farey-adjacent, it is
the unique Farey-adjacent pair with mediant $m/k$, and its fan-twist splice
(Theorem~\ref{thm:splice}) reproduces the prototile of $(m,k)$.
\item[\emph{(2)}] \emph{(Rigidity.)} Conversely, suppose that for some cut position
$1\le k_1'\le k-1$ and some reduced parameters $(m_1',k_1')$, $(m_2',k_2')$ with
$k_2'=k-k_1'$ (seeds allowed), the fan-twist splice of Theorem~\ref{thm:splice}---with
$(k_1,k_2)$ replaced by $(k_1',k_2')$---reproduces the lower path of $(m,k)$. Then
$k_1'=k_1$, and $m_i'=m_i$ whenever $k_i'\ge2$ (a block with $k_i'=1$ is realized by
either seed; see Remark~\ref{rem:seeds}).
\end{enumerate}
In particular, the prototile admits exactly one fan-twist splice decomposition, and a
fan-twist splice of two prototiles of the family reproduces a prototile of the family
\emph{only if} the parameters are Farey-adjacent.
\end{theorem}

\begin{proof}
(1) $m_1$ is an integer by the definition of the modular inverse, and
\[
m_2k_1-m_1k_2=(m-m_1)k_1-m_1(k-k_1)=k_1m-m_1k=1,
\]
so the pair is Farey-adjacent; $\gcd(m_i,k_i)=1$ follows from the determinant being $1$,
and Theorem~\ref{thm:splice} applied to the pair reproduces the prototile. For
uniqueness among Farey-adjacent pairs: if $(m_1,k_1),(m_2,k_2)$ are Farey-adjacent with
mediant $m/k$, then \eqref{eq:star} gives $k_1m\equiv 1 \pmod k$; since $0<k_1<k$,
necessarily $k_1=m^{-1}\bmod k$, and then $m_1=(mk_1-1)/k$ is forced.

(2) All edges are unit, so the splice reproduces the lower path if and only if it matches
every edge direction. Writing $s'^{(1)},s'^{(2)}$ for the direction indices of
$(m_1',k_1'),(m_2',k_2')$, the angle-matching conditions read
\[
\frac{2\pi s_j}{n}=\frac{2\pi s'^{(1)}_j}{t k_1'}+j\cdot\frac{2\pi}{k_1' n}
\quad(0\le j< k_1'),
\qquad
\frac{2\pi s_{k_1'+i}}{n}=\frac{2\pi s'^{(2)}_i}{t k_2'}+\frac{2\pi}{n}
-i\cdot\frac{2\pi}{k_2' n}
\quad(0\le i< k_2'),
\]
i.e., clearing denominators with $n=tk$,
\begin{equation}
k\,s'^{(1)}_j = k_1'\,s_j - j,
\qquad
k\,s'^{(2)}_i = k_2'\,(s_{k_1'+i}-1) + i.
\tag{$\dagger$}\label{eq:rigid}
\end{equation}
In particular $k$ must divide both right-hand sides.

\emph{The cut position is forced.} Take the last edge of the second block,
$i=k_2'-1$, so that $k_1'+i=k-1$ and $s_{k-1}=k-m$. The right-hand side of the second
equation in \eqref{eq:rigid} is
\[
k_2'(k-m-1)+(k_2'-1) \;=\; k_2'(k-m)-1 \;\equiv\; -(k_2'\,m+1) \pmod k,
\]
so divisibility forces $k_2'\,m\equiv-1\pmod k$. With $1\le k_2'\le k-1$ this gives
$k_2'=k-(m^{-1}\bmod k)$, i.e.\ $k_1'=m^{-1}\bmod k=k_1$.

\emph{The parameters are forced.} If $k_1\ge2$, take $j=1$ in \eqref{eq:rigid}: since
$s_1=m$, we get $s'^{(1)}_1=(k_1m-1)/k=m_1$, and $s'^{(1)}_1=m_1'\bmod k_1'=m_1'$,
whence $m_1'=m_1$. If $k_2\ge2$ then $m\le k-2$ (otherwise $m=k-1$ and $k_2=1$), so
$s_{k_1+1}=m+1$, and $i=1$ in \eqref{eq:rigid} gives, using \eqref{eq:star},
\[
k\,s'^{(2)}_1 = k_2\,m+1 = km-(k_1m-1) = k\,(m-m_1),
\]
whence $m_2'=s'^{(2)}_1=m-m_1=m_2$. Finally, a block with $k_i'=1$ has direction
sequence $(0)$ for both seeds, so \eqref{eq:rigid} imposes no further constraint on that
block (Remark~\ref{rem:seeds}).
\end{proof}

\begin{remark}[Degenerate blocks]\label{rem:seeds}
A block with $k_i'=1$ is a single unit edge of angle $0$ under either seed, $(0,1)$ or
$(1,1)$, so the geometry cannot distinguish the two seeds and
Theorem~\ref{thm:decompose}(2) cannot force $m_i'$. Requiring the pair to be
Farey-adjacent---equivalently $m_1'+m_2'=m$---selects the canonical value
$m_i'=m_i\in\{0,1\}$; this is the convention used in Corollary~\ref{cor:generation}
below.
\end{remark}

\begin{corollary}[Generation from a single edge]\label{cor:generation}
Iterating Theorem~\ref{thm:decompose} descends the Stern--Brocot tree and terminates,
after finitely many steps, at the seeds $(0,1)$ and $(1,1)$, whose lower words are both
$\ell=(0,1)$: geometrically one and the same unit rhombus (degenerate for $t=2$), and,
at the level of cores, a single unit edge of angle $0$. Reading the recursion upward,
the lower path of every prototile of the family is obtained from one unit edge by
iterated fan-twist splicing; since the upper word $u$ is determined by $\ell$, so is
the whole prototile. The leaves of the decomposition tree, read from left to right with
$(0,1)\mapsto 0$ and $(1,1)\mapsto 1$, spell the carry word $c(m,k)$.
\end{corollary}

\begin{proof}
For $k\ge2$ the cut satisfies $1\le k_1\le k-1$, so both denominators strictly decrease;
the recursion reaches $k=1$ in finitely many steps. Reversing it, each step is a splice
(Theorem~\ref{thm:splice}). The statement about leaves follows by applying
Corollary~\ref{cor:standard} along the tree, with base cases $c(0,1)=0$, $c(1,1)=1$.
\end{proof}

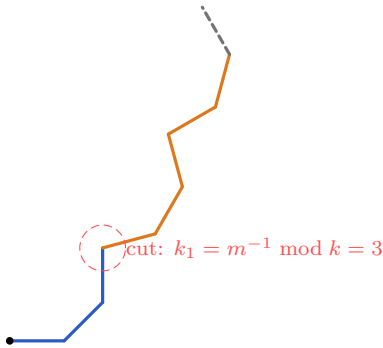
\begin{figure}[t]
\centering
\input{figs/fig-cut}
\caption{Decomposition of $(3,8)$: the cut position is $k_1=3^{-1}\bmod 8=3$.}
\label{fig:cut}
\end{figure}

\section{Variants and the separation of generation and decoration}\label{sec:variants}

The Bridges paper \cite{Imura2026} introduces variants $\mathrm{m},\mathrm{p},\mathrm{q}$
of the prototile, each consisting of two mirror-image parts. Table~\ref{tab:templates}
reproduces the boundary words of \cite[Table~1]{Imura2026} in the following token
notation: $\core_j$ denotes the core letter of index $j$ (direction $s_j$), and the two
\emph{decoration tokens} are $\Zt$ (direction $0$, angle $0$) and $\Kt$ (direction $k$,
angle $2\pi k/n = 2\pi/t$). Figure~\ref{fig:sep} shows the common core and three
decorated lower paths for the running example $(3,8)$.

\begin{table}[ht]
\centering
\small
\setlength{\tabcolsep}{1.4em}
\begin{tabular}{lll}
\toprule
template & lower word $\ell$ & upper word $u$\\
\midrule
original & $\core_0\cdots \core_{k-1}\;\Kt$ & $\Kt\;\core_1\cdots \core_{k-1}\;\Zt$\\
\addlinespace
$\mathrm{m}$, part 1 & $\core_0\cdots \core_{k-1}$ & $\core_1\cdots \core_{k-1}\;\Zt$\\
$\mathrm{m}$, part 2 & $\core_1\cdots \core_{k-1}\;\Kt$ & $\Kt\;\core_1\cdots \core_{k-1}$\\
\addlinespace
$\mathrm{p}$, part 1 & $\core_0\cdots \core_{k-1}\;\Kt$ & $\Kt\;\Zt\;\core_1\cdots \core_{k-1}$\\
$\mathrm{p}$, part 2 & $\core_0\cdots \core_{k-1}\;\Kt$ & $\core_1\cdots \core_{k-1}\;\Kt\;\Zt$\\
\addlinespace
$\mathrm{q}$, part 1 & $\core_1\cdots \core_{k-1}\;\Zt\;\Kt$ & $\Kt\;\core_1\cdots \core_{k-1}\;\Zt$\\
$\mathrm{q}$, part 2 & $\Zt\;\Kt\;\core_1\cdots \core_{k-1}$ & $\Kt\;\core_1\cdots \core_{k-1}\;\Zt$\\
\bottomrule
\end{tabular}
\caption{The seven boundary templates (after \cite[Table~1]{Imura2026}).
Every row has the shape (decorations)$^*$ (contiguous core run) (decorations)$^*$,
the core run being $\core_0\cdots\core_{k-1}$ or $\core_1\cdots\core_{k-1}$.}
\label{tab:templates}
\end{table}

\Needspace*{4\baselineskip}
\begin{theorem}[Separation of generation and decoration]\label{thm:separation}
\leavevmode
\begin{enumerate}
\item[\emph{(1)}] \emph{Core generation.} For every reduced $(m,k)$, the core edge
sequence (angles $2\pi s_j/n$, $j=0,\dots,k-1$) is reconstructed exactly from the leaves
of the decomposition tree---single unit edges of angle $0$---by iterated fan-twist
splicing (the core part of Theorem~\ref{thm:splice}).
\item[\emph{(2)}] \emph{Decoration attachment.} Define a \emph{template} to be a
boundary word of the shape
\[
(\text{decorations})^*\;\;
(\text{contiguous core run }\core_0\cdots\core_{k-1}\text{ or }\core_1\cdots\core_{k-1})
\;\;(\text{decorations})^*,
\]
where a
\emph{decoration token} is an edge whose direction is a multiple of $k$---angle
$2\pi c/t$ with $c\in\{0,\dots,t-1\}$, invariant across parents and child---and call a
pair $(\ell,u)$ of templates \emph{closed} if, counting $\core_0$ as the letter $\Zt$
(both have angle $0$), the decoration vector sums of $\ell$ and $u$ agree. For every closed template pair, mediant splicing lifts in the following sense.
Splice the \emph{full} cores $\core_0\cdots\core_{k_i-1}$ of the two parents as in
part~(1)---using a parent's $\core_0$ even when the template's core run omits it---and
place the decorations of the template around the corresponding run of the resulting
child core, leaving them unchanged. The result is the child part of the same template
pair. Explicitly, for the run $\core_1\cdots\core_{k-1}$ the child's core edges are the
edges $\core_1\cdots\core_{k_1-1}$ of the first parent and $\core_0\cdots\core_{k_2-1}$
of the second, rotated by the fan-twist amounts of Theorem~\ref{thm:splice}; the
second parent's $\core_0$ is therefore needed although it is not on the boundary of
that parent's part, so the operation is defined on full cores and not on the visible
boundary words of the parent parts. Every row of
Table~\ref{tab:templates} is a closed pair with decorations $\Zt,\Kt$; there $\ell$
and $u$ even use one multiset of directions.
\end{enumerate}
Consequently every prototile part of every variant is obtained as ``one unit edge
$\to$ (iterated splicing) $\to$ core $\to$ (finitely many decoration edges)'', and
mediant splicing lifts to every closed template pair, in particular to all seven
templates of Table~\ref{tab:templates}. Corollary~\ref{cor:generation} is the
special case of the original template (a single decoration $\Kt$).
\end{theorem}

\begin{proof}
(1) is Theorem~\ref{thm:decompose} (finite descent) combined with the core part of
Theorem~\ref{thm:splice} applied upward from the leaves; a leaf has $k=1$, core word
$(s_0)=(0)$ and angle $2\pi\cdot 0/t=0$.
(2) Compare the spliced part with the child part edge by edge. By part (1) and the
closed forms of Theorem~\ref{thm:splice}, the spliced full core is the child's full
core $\core_0\cdots\core_{k-1}$, so the core run $\core_0\cdots\core_{k-1}$ agrees; the
run $\core_1\cdots\core_{k-1}$ is obtained from it by deleting the first edge, which
comes from $\core_0$ of the first parent, and it agrees as well. A decoration edge
of angle $2\pi c/t$ represents the same vector in the systems of $(m_1,k_1)$,
$(m_2,k_2)$ and $(m,k)$, since $2\pi\,ck_1/n_1=2\pi c/t=2\pi\,ck/n$; inserted at the
same template position, it is unchanged by the splice. Finally each part closes: the
difference of the endpoints of $\ell$ and $u$ equals the difference of their decoration
vector sums (counting $\core_0$ as $\Zt$, the remaining core runs
$\core_1\cdots\core_{k-1}$ coincide), which vanishes by hypothesis---in every system,
the decoration vectors being system-invariant. For the rows of
Table~\ref{tab:templates} the hypothesis holds because $\ell$ and $u$ are permutations
of one multiset of directions.
\end{proof}

\begin{figure}[t]
\centering
\input{figs/fig-sep}
\caption{Separation of generation and decoration for $(3,8)$, $t=3$: the common core
(left) and three decorated lower paths; decorations dashed.}
\label{fig:sep}
\end{figure}
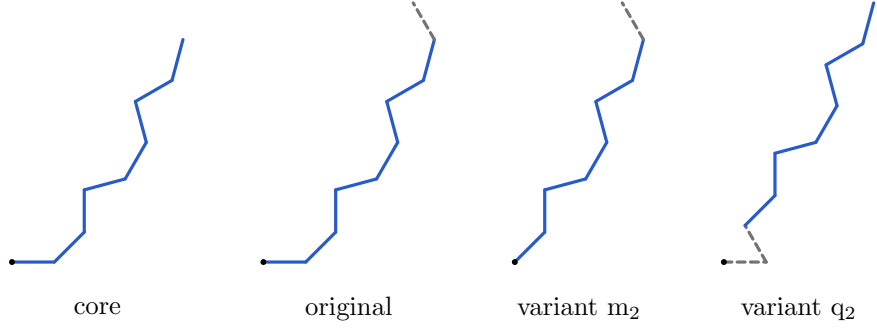

\begin{remark}[The decoration alphabet has $t$ letters]\label{rem:alphabet}
An edge direction $d$ (angle $2\pi d/n$ in the system of $(m,k)$) has system-invariant
angle iff $d$ is a multiple of $k$, i.e.\ iff the angle is $2\pi c/t$ for some
$c\in\{0,\dots,t-1\}$. Thus the decoration alphabet admitted by
Theorem~\ref{thm:separation}(2) has exactly $t$ letters, of which $\Zt$ ($c=0$) and
$\Kt$ ($c=1$) are the first two; we write $\Wt$ ($c=2$) and $\Xt$ ($c=3$) for the next
two. The closedness hypothesis constrains only the decoration \emph{vector sums} of
$\ell$ and $u$---strictly weaker than multiset equality (e.g.\ for $t=3$ the three
letters sum to zero)---and neither the number of decorations, nor their placement at
the two ends, nor their (fixed) lengths. Hence the family of closed template pairs
satisfying Theorems~\ref{thm:splice} and~\ref{thm:separation} is infinite. Tileability,
however, is not implied by this formal structure: it imposes the additional constraints
of simplicity and edge-matching between neighboring tiles, and which templates tile is
the subject of Problem~3 in Section~\ref{sec:open}. The gap is genuine: moving
\emph{both} decorations to opposite ends,
$\ell=\core_1\cdots\core_{k-1}\,\Zt\,\Kt$ and $u=\Kt\,\Zt\,\core_1\cdots\core_{k-1}$,
gives a multiset-equal (hence closed) template whose signed area is the \emph{constant}
$\sin(2\pi/t)$ for every $(m,k)$ (the core--decoration cross terms in the shoelace
formula cancel by the symmetry $s_j\leftrightarrow k-s_j$), whereas the original
template has area $\sin(2\pi/t)+2\sum_{j=1}^{k-1}\sin(2\pi j/n)$. For $t=2$ the
constant is $\sin\pi=0$: the template degenerates to zero area and is not a tile at
all, which proves the gap. For $t\ge3$ the failure persists in every case we
computed: for all reduced $(m,k)$ with $k\le40$ and all $t\in\{3,4,5\}$, the two
paths interlace and the boundary fails to be simple whenever $k\ge3$
(Section~\ref{sec:verification}); only $k=2$ yields a tile (a centrally symmetric
hexagon, tiling periodically).
\end{remark}

\section{Spiral arms and the visibility of the Farey parents}\label{sec:arms}

For the offset-free tiling, \cite{Imura2026} states that the numbers of counterclockwise
and clockwise spiral arms are $t\cdot(m^{-1}\bmod k)$ and $t\cdot(k-(m^{-1}\bmod k))$.

\begin{corollary}\label{obs:arms}
Granting this arm-count formula, the two counts are $t\,k_1$ and $t\,k_2$, where
$k_1$ and $k_2$ are the denominators of the Farey parents of $m/k$
(Theorem~\ref{thm:decompose}): the two Farey parents, scaled by $t$, are visible in
the tiling as its two families of spiral arms.
\end{corollary}

\begin{proof}
By Theorem~\ref{thm:decompose}(1), $k_1=m^{-1}\bmod k$ and $k_2=k-k_1$.
\end{proof}

The only input beyond this paper is the arm-count formula itself, which
\cite{Imura2026} states without proof. The corollary identifies the Farey parents at
the level of the tiling through these counts; a composition of the parent tilings into
the child tiling is a different matter (Problem~1 of Section~\ref{sec:open}).

\section{Computational experiments}\label{sec:verification}

Two claims in this paper rest on computation rather than proof. The scripts
(Node.js) are provided as ancillary files, and the counts below are printed by them:
the first paragraph by \texttt{verify-beta.mjs}, the second by
\texttt{verify-mediant-variants.mjs}. (A third computational finding, the tiling templates announced in Problem~3
of Section~\ref{sec:open}, will be documented in a companion paper.)

\paragraph{Rigidity of the direction-system size.}
For every reduced $(m,k)$ with $k\le5$ and every $k+2\le n\le 40$ ($315$
configurations: $\varphi(2)+\dots+\varphi(5)=9$ pairs, each with $39-k$ values of $n$,
$37+2\cdot36+2\cdot35+4\cdot34=315$), we ran the wedge construction of \cite{Imura2025} with three
closure rules---the full spiral of $n$ wedges, $g$-fold rotation about the origin,
and the point reflection of the offset case---and tested the result for overlaps
and gaps. A tiling is produced exactly when $n=tk$ or $n=2(tk-m)$ for some integer
$t\ge2$ ($121$ tilings, no exceptions). This is the evidence for the conjecture in
Problem~4 of Section~\ref{sec:open}.

\paragraph{Non-simplicity of the constant-area template.}
For every reduced $(m,k)$ with $k\le40$ and every $t\in\{3,4,5\}$ ($1467$ cases:
$\varphi(2)+\dots+\varphi(40)=489$ pairs times $3$ values of $t$),
the boundary of the constant-area template of Remark~\ref{rem:alphabet} is simple
only for $k=2$: for $k\ge3$ the two paths always intersect. We have not proved this
for all $(k,t)$; the counterexample role of that template rests on the proven case
$t=2$ alone.

\section{Open problems}\label{sec:open}

\begin{enumerate}
\item \textbf{Tiling-level composition.} Our operation is established at the level of
prototile shapes. Is there a direct composition of the two parent \emph{tilings} into the
child tiling? Imura's front-tracking construction \cite{Imura2025} and
Corollary~\ref{obs:arms} suggest that the child's front should decompose into
interleaved copies of the parents' fronts.
\item \textbf{Variant tilings.} The placement rule of the two mirror parts inside a
wedge is given in \cite{Imura2026} pictorially; it is implemented in Imura's public
interactive generator\footnote{\url{https://mk.tiling.jp/playground}}. In work
subsequent to this paper we extracted and formalized the rule---each variant wedge
is a union of two triangular cones, one per part, cut out of translation lattices
of the single parts, positioned by a closed-form wedge origin---and this will be
documented in the companion paper of Problem~3. Whether the formalization extends
Theorem~\ref{thm:separation} from prototiles to variant \emph{tilings}, in the
sense of Problem~1, remains open.
\item \textbf{Exotic decorations.} Which of the infinitely many closed template pairs
of Theorem~\ref{thm:separation}(2) tile the plane? In computational experiments
subsequent to this paper we found mirror pairs of templates using the third letter
$\Wt$ (so $t\ge3$) that do tile---periodically, on sheared lattices, and
non-periodically, filling the wedge structure of \cite{Imura2026}; details will appear
in a companion paper. A classification of the tiling templates, and rigorous tiling
proofs, remain open.
\item \textbf{Rigidity of the direction-system size.} Remark~\ref{rem:offset} extends
the splice, with closed-form corrections, to every affine convention
$n=\alpha k+\beta m$. Tileability appears to be far more rigid: the scan reported in
Section~\ref{sec:verification} finds that the wedge construction of \cite{Imura2025}
produces a tiling exactly when
$n=tk$ or $n=2(tk-m)$ for some integer $t\ge2$. We conjecture that the only
conventions tiling uniformly in $(m,k)$ are Imura's two, $(\alpha,\beta)=(t,0)$ and
$(2t,-2)$---in particular $\beta\in\{0,-2\}$ is necessary but not sufficient: the
scan covers, e.g., $n=3k-2m$, which has $\beta=-2$ with odd $\alpha$ and fails. (A
convention with other $(\alpha,\beta)$ can still tile at sporadic parameters where
$\alpha k+\beta m$ happens to coincide with an Imura value.) Thus the formal
Farey-splice structure of Remark~\ref{rem:offset} exists for every $(\alpha,\beta)$,
but only two conventions reach the plane. Whether the $(m,k,n)$-prototile for other
$n$ can tile by some arrangement unrelated to the wedge construction is open as
well.
\end{enumerate}

\paragraph{Acknowledgements.}
The author is deeply grateful to Miki Imura for her inspiring tech talks and
articles shared within the internal engineering community, which motivated this
follow-up study and laid its groundwork. The author thanks Gábor Damásdi for
valuable comments on an earlier version of this manuscript. The author acknowledges the use of
LLM-based applications for assistance in editing and refining the presentation
of the manuscript; all ideas and results presented in this paper are solely
those of the author.

\bibliographystyle{plain}
\bibliography{refs}

\end{document}

%% file: figs/fig-tree.tex
\begin{tikzpicture}[scale=0.8,line join=round]
\draw[pathblue!50,line width=0.5pt] (6.5,-2.05) -- (0,0);
\draw[pathorange!65,line width=0.5pt] (6.5,-2.05) -- (13,0);
\draw[pathblue!50,line width=0.5pt] (4.3333,-4.1) -- (0,0);
\draw[pathorange!65,line width=0.5pt] (4.3333,-4.1) -- (6.5,-2.05);
\draw[pathblue!50,line width=0.5pt] (8.6667,-4.1) -- (6.5,-2.05);
\draw[pathorange!65,line width=0.5pt] (8.6667,-4.1) -- (13,0);
\draw[pathblue!50,line width=0.5pt] (3.25,-6.15) -- (0,0);
\draw[pathorange!65,line width=0.5pt] (3.25,-6.15) -- (4.3333,-4.1);
\draw[pathblue!50,line width=0.5pt] (5.2,-6.15) -- (4.3333,-4.1);
\draw[pathorange!65,line width=0.5pt] (5.2,-6.15) -- (6.5,-2.05);
\draw[pathblue!50,line width=0.5pt] (7.8,-6.15) -- (6.5,-2.05);
\draw[pathorange!65,line width=0.5pt] (7.8,-6.15) -- (8.6667,-4.1);
\draw[pathblue!50,line width=0.5pt] (9.75,-6.15) -- (8.6667,-4.1);
\draw[pathorange!65,line width=0.5pt] (9.75,-6.15) -- (13,0);
\draw[pathblue!50,line width=0.5pt] (4.875,-8.2) to[bend left=14] (4.3333,-4.1);
\draw[pathorange!65,line width=0.5pt] (4.875,-8.2) -- (5.2,-6.15);
\fill[fill=black!6,draw=black!60,line width=0.5pt] (-0.2167,-0.3753)--(0.65,-0.3753)--(0.2167,0.3753)--(-0.65,0.3753)--cycle;
\node[below=0.5pt] at (0,-0.3753) {\scriptsize $\tfrac{0}{1}$};
\fill[fill=black!6,draw=black!60,line width=0.5pt] (12.7833,-0.3753)--(13.65,-0.3753)--(13.2167,0.3753)--(12.35,0.3753)--cycle;
\node[below=0.5pt] at (13,-0.3753) {\scriptsize $\tfrac{1}{1}$};
\fill[fill=black!6,draw=black!60,line width=0.5pt] (6.175,-2.6129)--(6.825,-2.6129)--(7.15,-2.05)--(6.825,-1.4871)--(6.175,-1.4871)--(5.85,-2.05)--cycle;
\node[below=0.5pt] at (6.5,-2.6129) {\scriptsize $\tfrac{1}{2}$};
\fill[fill=black!6,draw=black!60,line width=0.5pt] (3.9581,-4.75)--(4.4794,-4.75)--(4.8787,-4.4149)--(4.9693,-3.9015)--(4.7086,-3.45)--(4.1873,-3.45)--(4.0968,-3.9634)--(3.6974,-4.2985)--cycle;
\node[below=0.5pt] at (4.3333,-4.75) {\scriptsize $\tfrac{1}{3}$};
\fill[fill=black!6,draw=black!60,line width=0.5pt] (8.2914,-4.75)--(8.8127,-4.75)--(8.9032,-4.2366)--(9.3026,-3.9015)--(9.0419,-3.45)--(8.5206,-3.45)--(8.1213,-3.7851)--(8.0307,-4.2985)--cycle;
\node[below=0.5pt] at (8.6667,-4.75) {\scriptsize $\tfrac{2}{3}$};
\fill[fill=black!6,draw=black!60,line width=0.5pt] (2.8747,-6.8)--(3.2769,-6.8)--(3.6253,-6.5989)--(3.8264,-6.2506)--(3.8264,-5.8483)--(3.6253,-5.5)--(3.2231,-5.5)--(3.2231,-5.9022)--(3.0219,-6.2506)--(2.6736,-6.4517)--cycle;
\node[below=0.5pt] at (3.25,-6.8) {\scriptsize $\tfrac{1}{4}$};
\fill[fill=black!6,draw=black!60,line width=0.5pt] (4.8247,-6.8)--(5.1529,-6.8)--(5.3725,-6.5561)--(5.3382,-6.2298)--(5.638,-6.0963)--(5.7394,-5.7842)--(5.5753,-5.5)--(5.2471,-5.5)--(5.1457,-5.8121)--(4.8459,-5.9456)--(4.8802,-6.2719)--(4.6606,-6.5158)--cycle;
\node[below=0.5pt] at (5.2,-6.8) {\scriptsize $\tfrac{2}{5}$};
\fill[fill=black!6,draw=black!60,line width=0.5pt] (7.4247,-6.8)--(7.7529,-6.8)--(7.8543,-6.4879)--(8.1541,-6.3544)--(8.1198,-6.0281)--(8.3394,-5.7842)--(8.1753,-5.5)--(7.8471,-5.5)--(7.6275,-5.7439)--(7.6618,-6.0702)--(7.362,-6.2037)--(7.2606,-6.5158)--cycle;
\node[below=0.5pt] at (7.8,-6.8) {\scriptsize $\tfrac{3}{5}$};
\fill[fill=black!6,draw=black!60,line width=0.5pt] (9.3747,-6.8)--(9.7769,-6.8)--(9.7769,-6.3978)--(9.9781,-6.0494)--(10.3264,-5.8483)--(10.1253,-5.5)--(9.7231,-5.5)--(9.3747,-5.7011)--(9.1736,-6.0494)--(9.1736,-6.4517)--cycle;
\node[below=0.5pt] at (9.75,-6.8) {\scriptsize $\tfrac{3}{4}$};
\fill[fill=pathblue!12,draw=pathblue,line width=0.7pt] (4.4997,-8.85)--(4.7118,-8.85)--(4.8618,-8.7)--(4.8618,-8.488)--(5.0666,-8.4331)--(5.1727,-8.2494)--(5.1178,-8.0446)--(5.3014,-7.9385)--(5.3563,-7.7337)--(5.2503,-7.55)--(5.0382,-7.55)--(4.9833,-7.7549)--(4.7996,-7.8609)--(4.8545,-8.0657)--(4.7485,-8.2494)--(4.5436,-8.3043)--(4.5436,-8.5164)--(4.3937,-8.6663)--cycle;
\node[below=0.5pt] at (4.875,-8.85) {\scriptsize $\tfrac{3}{8}$};
\end{tikzpicture}

%% file: figs/fig-proto.tex
\begin{tikzpicture}[scale=1.05,line cap=round,line join=round]
\draw[fill=black!6,draw=none] (0,0) -- (1,0) -- (1.6691,0.7431) -- (1.5646,1.7376) -- (2.4781,2.1443) -- (2.7871,3.0954) -- (2.2871,3.9614) -- (1.2871,3.9614) -- (0.9781,3.0103) -- (0.0646,2.6036) -- (0.1691,1.6091) -- (-0.5,0.866) -- cycle;
\draw[black!45,thick,densely dashed] (0,0) -- (-0.5,0.866);
\draw[black!45,thick] (-0.5,0.866) -- (0.1691,1.6091);
\draw[black!45,thick] (0.1691,1.6091) -- (0.0646,2.6036);
\draw[black!45,thick] (0.0646,2.6036) -- (0.9781,3.0103);
\draw[black!45,thick] (0.9781,3.0103) -- (1.2871,3.9614);
\draw[black!45,thick] (1.2871,3.9614) -- (2.2871,3.9614);
\draw[pathblue,very thick] (0,0) -- (1,0);
\draw[pathblue,very thick] (1,0) -- (1.6691,0.7431);
\draw[pathblue,very thick] (1.6691,0.7431) -- (1.5646,1.7376);
\draw[pathblue,very thick] (1.5646,1.7376) -- (2.4781,2.1443);
\draw[pathblue,very thick] (2.4781,2.1443) -- (2.7871,3.0954);
\draw[black!55,very thick,densely dashed] (2.7871,3.0954) -- (2.2871,3.9614);
\fill (0,0) circle (1.6pt) node[below left=-1pt] {\scriptsize $P$};
\fill (2.2871,3.9614) circle (1.6pt) node[above right=-1pt] {\scriptsize $Q$};
\node[pathblue,below right=2pt] at (1.6691,0.7431) {\scriptsize $L$};
\node[black!55,above left=1pt] at (0.1691,1.6091) {\scriptsize $U$};
\end{tikzpicture}

%% file: figs/fig-splice.tex
\begin{tikzpicture}[scale=0.62,line cap=round,line join=round]
\draw[pathblue,very thick] (0,0) -- (1,0);
\draw[pathblue,very thick] (1,0) -- (1.766,0.6428);
\draw[pathblue,very thick] (1.766,0.6428) -- (1.9396,1.6276);
\draw[black!55,very thick,densely dashed] (1.9396,1.6276) -- (1.4396,2.4936);
\fill (0,0) circle (2pt);
\node at (0.9698,-1.1) {\small $(1,3)$};
\draw[pathorange,very thick] (4.0396,0) -- (5.0396,0);
\draw[pathorange,very thick] (5.0396,0) -- (5.7087,0.7431);
\draw[pathorange,very thick] (5.7087,0.7431) -- (5.6042,1.7376);
\draw[pathorange,very thick] (5.6042,1.7376) -- (6.5177,2.1443);
\draw[pathorange,very thick] (6.5177,2.1443) -- (6.8267,3.0954);
\draw[black!55,very thick,densely dashed] (6.8267,3.0954) -- (6.3267,3.9614);
\fill (4.0396,0) circle (2pt);
\node at (5.4332,-1.1) {\small $(2,5)$};
\draw[pathblue,very thick] (8.9267,0) -- (9.9267,0);
\draw[pathblue,very thick] (9.9267,0) -- (10.6338,0.7071);
\draw[pathblue,very thick] (10.6338,0.7071) -- (10.6338,1.7071);
\draw[pathorange,very thick] (10.6338,1.7071) -- (11.5997,1.9659);
\draw[pathorange,very thick] (11.5997,1.9659) -- (12.0997,2.8319);
\draw[pathorange,very thick] (12.0997,2.8319) -- (11.8409,3.7978);
\draw[pathorange,very thick] (11.8409,3.7978) -- (12.7069,4.2978);
\draw[pathorange,very thick] (12.7069,4.2978) -- (12.9657,5.2637);
\draw[black!55,very thick,densely dashed] (12.9657,5.2637) -- (12.4657,6.1297);
\fill (8.9267,0) circle (2pt);
\node at (10.9462,-1.1) {\small $(3,8)$};
\node at (2.9896,2.2) {$\oplus$};
\node at (7.8767,2.2) {$=$};
\end{tikzpicture}

%% file: figs/fig-cut.tex
\begin{tikzpicture}[scale=0.72,line cap=round,line join=round]
\draw[pathblue,very thick] (0,0) -- (1,0);
\draw[pathblue,very thick] (1,0) -- (1.7071,0.7071);
\draw[pathblue,very thick] (1.7071,0.7071) -- (1.7071,1.7071);
\draw[pathorange,very thick] (1.7071,1.7071) -- (2.673,1.9659);
\draw[pathorange,very thick] (2.673,1.9659) -- (3.173,2.8319);
\draw[pathorange,very thick] (3.173,2.8319) -- (2.9142,3.7978);
\draw[pathorange,very thick] (2.9142,3.7978) -- (3.7802,4.2978);
\draw[pathorange,very thick] (3.7802,4.2978) -- (4.039,5.2637);
\draw[black!55,very thick,densely dashed] (4.039,5.2637) -- (3.539,6.1297);
\fill (0,0) circle (2pt);
\draw[red!70,densely dashed] (1.7071,1.7071) circle (0.42);
\node[red!70,right=5pt] at (1.7071,1.7071) {\scriptsize cut: $k_1 = m^{-1} \bmod k = 3$};
\end{tikzpicture}

%% file: figs/fig-sep.tex
\begin{tikzpicture}[scale=0.56,line cap=round,line join=round]
\draw[pathblue,very thick] (0,0) -- (1,0);
\draw[pathblue,very thick] (1,0) -- (1.7071,0.7071);
\draw[pathblue,very thick] (1.7071,0.7071) -- (1.7071,1.7071);
\draw[pathblue,very thick] (1.7071,1.7071) -- (2.673,1.9659);
\draw[pathblue,very thick] (2.673,1.9659) -- (3.173,2.8319);
\draw[pathblue,very thick] (3.173,2.8319) -- (2.9142,3.7978);
\draw[pathblue,very thick] (2.9142,3.7978) -- (3.7802,4.2978);
\draw[pathblue,very thick] (3.7802,4.2978) -- (4.039,5.2637);
\fill (0,0) circle (2pt);
\node at (2.0195,-1.1) {\small core};
\draw[pathblue,very thick] (5.939,0) -- (6.939,0);
\draw[pathblue,very thick] (6.939,0) -- (7.6461,0.7071);
\draw[pathblue,very thick] (7.6461,0.7071) -- (7.6461,1.7071);
\draw[pathblue,very thick] (7.6461,1.7071) -- (8.612,1.9659);
\draw[pathblue,very thick] (8.612,1.9659) -- (9.112,2.8319);
\draw[pathblue,very thick] (9.112,2.8319) -- (8.8532,3.7978);
\draw[pathblue,very thick] (8.8532,3.7978) -- (9.7192,4.2978);
\draw[pathblue,very thick] (9.7192,4.2978) -- (9.978,5.2637);
\draw[black!55,very thick,densely dashed] (9.978,5.2637) -- (9.478,6.1297);
\fill (5.939,0) circle (2pt);
\node at (7.9585,-1.1) {\small original};
\draw[pathblue,very thick] (11.878,0) -- (12.5851,0.7071);
\draw[pathblue,very thick] (12.5851,0.7071) -- (12.5851,1.7071);
\draw[pathblue,very thick] (12.5851,1.7071) -- (13.551,1.9659);
\draw[pathblue,very thick] (13.551,1.9659) -- (14.051,2.8319);
\draw[pathblue,very thick] (14.051,2.8319) -- (13.7922,3.7978);
\draw[pathblue,very thick] (13.7922,3.7978) -- (14.6582,4.2978);
\draw[pathblue,very thick] (14.6582,4.2978) -- (14.917,5.2637);
\draw[black!55,very thick,densely dashed] (14.917,5.2637) -- (14.417,6.1297);
\fill (11.878,0) circle (2pt);
\node at (13.3975,-1.1) {\small variant $\mathrm{m}_2$};
\draw[black!55,very thick,densely dashed] (16.817,0) -- (17.817,0);
\draw[black!55,very thick,densely dashed] (17.817,0) -- (17.317,0.866);
\draw[pathblue,very thick] (17.317,0.866) -- (18.0241,1.5731);
\draw[pathblue,very thick] (18.0241,1.5731) -- (18.0241,2.5731);
\draw[pathblue,very thick] (18.0241,2.5731) -- (18.99,2.8319);
\draw[pathblue,very thick] (18.99,2.8319) -- (19.49,3.6979);
\draw[pathblue,very thick] (19.49,3.6979) -- (19.2312,4.6638);
\draw[pathblue,very thick] (19.2312,4.6638) -- (20.0972,5.1638);
\draw[pathblue,very thick] (20.0972,5.1638) -- (20.356,6.1297);
\fill (16.817,0) circle (2pt);
\node at (18.5865,-1.1) {\small variant $\mathrm{q}_2$};
\end{tikzpicture}